\documentclass[11pt, twoside]{article}
\usepackage{amsfonts,amssymb,amsmath,amsthm}
\usepackage{authblk}
\usepackage{graphicx}
\usepackage[all]{xy}
\usepackage{tikz}
\usetikzlibrary{arrows.meta}
\usepackage{changepage}
\usepackage{psfrag,xmpmulti,amscd,color,pstricks, import,enumerate}
\usepackage[normalem]{ulem}

 \divide\oddsidemargin by 2
\newtheorem{thm}{Theorem}[section]

\newtheorem{lem}[thm]{Lemma}

\newtheorem{claim}[thm]{Claim}
\newtheorem*{thmA}{Theorem A}
\newtheorem*{thmB}{Theorem B}

\newtheorem*{thmC}{Theorem C}
\newtheorem*{thmD}{Theorem D}
\newtheorem*{thmE}{Theorem E}

\theoremstyle{definition}

\newtheorem{rem}[thm]{Remark}
\newtheorem*{rem*}{Remark}
\newtheorem*{rems*}{Remarks}

\newtheorem*{ex*}{Example}

\numberwithin{equation}{section}

\definecolor{OrangeRed}{cmyk}{0,0.6,1,0}            
\definecolor{DarkBlue}{cmyk}{1,1,0,0.20}
\definecolor{DarkGreen}{cmyk}{1,0,0.6,0.2}
\definecolor{myblue}{rgb}{0.66,0.78,1.00}
\definecolor{Violet}{cmyk}{0.79,0.88,0,0}
\definecolor{Lavender}{cmyk}{0,0.48,0,0}
\definecolor{purpleheart}{rgb}{0.41, 0.21, 0.61}
\definecolor{brick}{cmyk}{0,0.8,0.3,0.5}

\renewcommand{\AA}{{\cal A}}

\newcommand{\D}{{\mathbb D}}
\newcommand{\Hyp}{{\mathbb H}}

\newcommand{\N}{{\mathbb N}}

\newcommand{\R}{{\mathbb R}}

\newcommand{\la}{\lambda}

\newcommand{\eps}{\varepsilon}

\renewcommand{\epsilon}{\varepsilon}
\renewcommand{\epsilon}{\varepsilon}
\renewcommand{\phi}{\varphi}

\vspace{5cm}

\title{Non-autonomous dynamics of inner functions: mixing and entropy}

\author[3]{Vasiliki Evdoridou} 
\author[1,2]{N\'uria Fagella \thanks{ Grant PID2023-147252NB-I00 from Agencia Estatal de Investigaci\'on, the Severo Ochoa and María de Maeztu
Program for Centers and Units of Excellence in R\&D (CEX2020-001084-M), and ICREA Academia 2020. }}
\author[3]{Philip J. Rippon}
\author[3]{Gwyneth M. Stallard}
\affil[1]{\small Dep. de Matem\`atiques i Inform\`atica, Universitat de Barcelona, Catalonia, Spain.}
\affil[2]{\small Centre de Recerca Matemàtica, Bellaterra, Catalonia, Spain.}
\affil[3]{\small School of Mathematics and Statistics, The Open University, Milton Keynes, UK.}
\date{\today}

\begin{document}

\maketitle
\begin{abstract}
We generalise a mixing result due to Pommerenke for non-autonomous dynamical systems of forward compositions of the boundary values of centred inner functions; the proof requires a precise bound on the orbits of such forward compositions in~$\D$. This enables us to give new criteria for such systems to have positive metric entropy, lower entropy and topological entropy. In the opposite direction, we construct a forward composition of centred degree two Blaschke products with topological entropy zero.
\end{abstract}

\section{Introduction}
An \textit{inner function} is a function holomorphic in the unit disc $\D=\{z:|z|<1\}$ with non-tangential boundary values of modulus~1 almost everywhere on $\partial\D$. It is \textit{centred} if $0$ is a fixed point. Throughout the paper we denote the non-tangential boundary values of an inner function~$f$ by the same letter and we use the notation $|E|$, where $E \subset \partial \D$, to denote the Lebesgue one-dimensional measure of~$E$.

The following `strong mixing' result for the boundary values of sequences of centred inner functions is due to Pommerenke~\cite[Theorem~1 and Lemma~3]{pommerenke-ergodic}.

\begin{thm} \label{thm:Pom3}
Let $F_n = f_n \circ \cdots \circ f_1$, where $f_n$, $n\in \N$, are centred inner functions such that $|f'_n(0)|\le \lambda$ where $1/2\le\lambda<1.$ Then there is an absolute constant~$K$ such that
\begin{equation}\label{eq:expmixing}
\left| \frac{|A \cap F_n^{-1}(E)|}{|E|}- \frac{|A|}{2\pi}\right| \leq K c^n,\;\;\text{for }n\ge 1, \quad\text{where } c=\exp \left(-(1-\lambda)/84\right),
\end{equation}
for all arcs $A$ and measurable sets $E$ in $\partial \D$ with $|E|>0$.

If $F_n$, $n\in \N$, are any centred inner functions such that $F_n(z)\to 0$ as $n\to\infty$, for all $z\in\D$, then an inequality of the form \eqref{eq:expmixing} holds, with $o(1)$ as the upper bound.
\end{thm}

Pommerenke used the exponential mixing estimate \eqref{eq:expmixing} to obtain a positive result on the entropy of the system $(f_n)$ on $\partial \D$. More widely, Theorem~\ref{thm:Pom3} has had significant applications to the boundary dynamics of inner functions; see \cite{Fernandez, BEFRS2, BEFRS3, Bennett26}. For example, in \cite{BEFRS3} we used \eqref{eq:expmixing} to prove a sharp shrinking target result for forward compositions of centred inner functions, with an application to boundary dynamics of wandering domains of transcendental entire functions.

Centred inner functions are measure-preserving on $\partial \D$, that is, $|f^{-1}(E)|=|E|$ for all such functions $f$ and measurable sets~$E$ in $\partial \D$, and the forward compositions of centred inner functions in Theorem~\ref{thm:Pom3} form a non-autonomous dynamical system. The mixing results in Theorem~\ref{thm:Pom3} measure the extent to which preimages of sets under forward compositions of centred inner functions, or under sequences of such functions, are uniformly distributed on $\partial \D$; see \cite{FerreiraNic} for another approach to mixing for forward compositions of inner functions, related to an extension of the concept of ergodicity for such systems.

In this paper, we generalise Theorem~\ref{thm:Pom3} by relaxing the uniform bound on $|f'_n(0)|$, $n=1,2,\ldots,$ and we use this generalisation to obtain new criteria for the entropy of such non-autonomous dynamical systems to be positive.
\begin{thmA}\label{lem:3}
Let $f_n$, $n\in \N$, be centred inner functions such that $|f_n'(0)|\le \lambda_n$, where $0<a\le \lambda_n <1$ and $\lambda_n \cdots \lambda_1 \to 0$ as $n \to \infty$. Then $F_n= f_n \circ \cdots \circ f_1$ satisfies
\begin{equation}\label{eq:newmixing}
\left|\frac{|A \cap F_n^{-1}(E)|}{|E|} - \frac{|A|}{2\pi}\right| \le K(\lambda_n \cdots \lambda_1)^{\delta},
\end{equation}
for all arcs $A\subset \partial\D$ and measurable sets $E$ in $\partial \D$ with $|E|>0$, where $\delta=(6(1/a+1))^{-1}$ and $K=K(a)$ is a positive constant.
\end{thmA}
If we impose the constraint from Theorem~\ref{thm:Pom3} that $|f_n'(0)|\le \lambda$ where $1/2 \le \lambda<1$, then the upper bound in \eqref{eq:newmixing} is of the form $K d^n$, where $d=\lambda^{1/18}$ and $K$ is absolute; this constant~$d$ is somewhat less than the constant~$c$ in Theorem~\ref{thm:Pom3}.

In order to prove Theorem~\ref{thm:Pom3}, which concerns the backward behaviour of $(F_n)$ on $\partial\D$, Pommerenke had the insight to use the forward behaviour of $(F_n)$ within $\D$. He obtained the following estimate \cite[Lemma 2]{pommerenke-ergodic} for the rate of convergence to~0 of forward compositions of inner functions that fix~0 and have a uniformly bounded derivative at~0, which is of interest in its own right; see, for example, the application in \cite{Nicolau-Donaire}.
\begin{thm}\label{lem:Pom3}
Let $f_n$ be centred inner functions with $|f_n'(0)| \leq \lambda$ for all $n \in \N$, where $1/2\le\lambda <1$. Then $F_n = f_n \circ \cdots \circ f_1$ satisfies
\[
|F_n(z)|\le \frac{\lambda^n}{(1-|z|)^{13}},\quad\text{for } z\in \D.
\]
\end{thm}
The key feature of this estimate, which is needed to prove Theorem~\ref{thm:Pom3}, is that the growth as $|z|\to 1$ is a power of $1/(1-|z|)$. Pommerenke commented that the bound in this estimate is `surely not sharp', a statement that also applies to the bounds we give in this paper.

To prove Theorem~A we generalise Theorem~\ref{lem:Pom3} to a non-uniform setting, as follows.

\begin{thmB}\label{thm:B1}
For $n\ge 1$, let $f_n$ be centred inner functions such that $|f_n'(0)|\le \lambda_n$, where $0<a\le \lambda_n<1$ and $\lambda_n \cdots \lambda_1 \to 0$ as $n \to \infty$. Then $F_n= f_n \circ \cdots \circ f_1$ satisfies
\begin{equation}\label{eq:est2}
|F_n(z)| \leq K \lambda_n \cdots \lambda_1 \frac{|z|}{(1-|z|)^{p}} ,\quad \text{for} \; z \in \mathbb{D},
\end{equation}
where $p=2/a+ 1$ and $K=K(a)$ is a positive constant.

\end{thmB}

The proof of Theorem~B is considerably more complicated than that of Theorem~\ref{lem:Pom3} because of the lack of a uniform bound strictly less than~1 on the $\lambda_n$. Note that the constants $K=K(a)$ in \eqref{eq:est2} and \eqref{eq:newmixing} are exponentially large for values of~$a$ close to~0. These constants cannot be taken to be independent of~$a$ because, when $|f'(0)|$ is small, the contracting effect of an inner function~$f$ near the origin is much greater than at points close to $\partial\D$. However, the constants $K(a)$ and $\delta(a)$ can be taken to be absolute if we allow a product of larger $\lambda_n$ terms on the right; see Remark~\ref{bigger-lambda}.

In later work, Pommerenke studied such forward compositions of centred inner functions of the non-uniform type \cite{PommComps}. Among the results he proved was an upper bound for $|F_n(z)|$ of the form $M\lambda_n \cdots \lambda_1$, where~$M$ was an unspecified positive quantity, stated to depend only on~$a$ and~$|z|$. Inspection of the argument in \cite[Section~2]{PommComps} shows that it gives
\[
|F_n(z)|\le \lambda_n \cdots \lambda_1|z|\exp\left(\frac{4\log(1/a)}{a(1-|z|)}\right),\quad\text{for }z\in\D.
\]
Our proof of Theorem~B begins by finding a (slightly smaller) bound of this type, and then leverages our exponential bound to obtain \eqref{eq:est2}, which allows us to deduce Theorem~A.

Theorem~A allows us to generalise another result from \cite{pommerenke-ergodic} concerning the entropy of the boundary values of forward compositions of centred inner functions.

First, we introduce the notion of entropy defined in \cite{pommerenke-ergodic} for forward compositions $F_n=f_n\circ\cdots\circ f_1$ of centred inner functions $(f_n)$.   Let~$\AA$ denote the set of finite partitions of the unit circle. For such a partition $\{A_1,\cdots,A_N\}$ and $k_0, k_1, \ldots, k_n \in \{1,2, \ldots, N\}$, let
\begin{equation}\label{eq:Ak0kn}
A_{k_0\ldots k_n}=A_{k_0}\cap F_1^{-1}(A_{k_1})\cap \cdots\cap F_n^{-1}(A_{k_n}),
\end{equation}
which is the set of points in $\partial\D$ with initial itinerary $(k_0\ldots k_n)$. Pommerenke defined the {\em lower entropy} of the non-autonomous dynamical system $(f_n)$ to be
\begin{equation}\label{eq:lower-ent}
\underline{h}((f_n)):=
\sup_\AA \liminf_{n\to\infty}\frac{1}{n+1} \sum_{k_0=1}^N \ldots \sum_{k_n=1}^N
\vert A_{k_0\ldots k_n}\vert \log \frac{2\pi}{\vert A_{k_0\ldots k_n}\vert}.
\end{equation}
For an autonomous system, where $f_n=f$ for all~$n$, the lower entropy coincides with the measure theoretic entropy $h(f)$, also called the metric entropy; see \cite[Chapter~4]{Walters}.

Pommerenke gave the following lower bound for the lower entropy of the non-autonomous system $(f_n)$ under the condition of uniform contraction \cite[Theorem~2 and Corollary]{pommerenke-ergodic}.
\begin{thm}\label{Pomm-ent}
Let $f_n$ be centred inner functions with $\vert f_n'(0)\vert \leq \lambda <1$, for $n=1,2,\ldots$. Then, the lower entropy $\underline{h}$ of $(f_n)$ satisfies
\[
\underline{h}((f_n)) \geq c(1-\lambda),
\]
where $c>0$ is an absolute constant.

In particular, if~$f$ is a centred inner function, then $h(f)\geq c(1-|f'(0)|)$.
\end{thm}

More recently, Kawan \cite[Section~3.1]{Kawan} introduced the concept of \textit{metric entropy} for general non-autonomous dynamical systems. In our setting, this quantity is defined as in~\eqref{eq:lower-ent} but with $\limsup_{n\to\infty}$ rather than $\liminf_{n\to\infty}$, and is denoted by $h((f_n))$; thus $h((f_n))\ge \underline{h}((f_n))$ in general.

We use Theorem~A to generalise Pommerenke's result on lower entropy by removing the requirement for $(f_n)$ to be \textit{uniformly} contracting; that is, we allow $|f_n'(0)|$ to take values arbitrarily close to~1. We also give a sufficient condition for positive metric entropy.
\begin{thmC}
Let $f_n$ be centred inner functions with $\la_n=\vert f_n'(0)\vert \geq a>0$ and $\lambda_n \cdots \lambda_1 \to 0$ as $n \to \infty$, and put $\mu_n:=1-\la_n$ for n=1,2,\ldots.
\begin{itemize}
\item[(a)] If
\begin{equation}\label{eq:lower-entropy}
\limsup_{n\to\infty} \left(\la_n \cdots\la_1\right)^{1/n} < 1; \quad\text{that is,}\quad \underline\mu:=\liminf_{n\to\infty}\frac{1}{n}\sum_{k=1}^n \mu_k>0,
\end{equation}
then the lower entropy $\underline{h}$ of $(f_n)$ satisfies
\[
\underline{h}((f_n)) \geq c\underline\mu,
\]
where $c$ is a positive constant that depends only on~$a$.
\item[(b)] If
\begin{equation}\label{eq:entropy}
\liminf_{n\to\infty} \left(\la_n \cdots\la_1\right)^{1/n} < 1; \quad\text{that is,}\quad \mu:=\limsup_{n\to\infty}\frac{1}{n}\sum_{k=1}^n \mu_k>0,
\end{equation}
then the metric entropy~$h$ of~$(f_n)$ satisfies
\[
h((f_n)) \geq c\mu,
\]
where~$c$ is a positive constant that depends only on~$a$.
\end{itemize}
\end{thmC}
Note that the condition $\lim_{n\to\infty} \la_n \cdots\la_1 =0$ is equivalent to $\sum_{n=1}^\infty \mu_n =\infty$. Also, note that the statements that \eqref{eq:lower-entropy} and \eqref{eq:entropy} are both equivalences depends on the hypothesis that the terms $\lambda_n$ are bounded away from zero.

It is natural to ask to what extent the results in Theorem~C are best possible; that is, can the additional hypotheses on $\lambda_n$ in the two parts be weakened or perhaps omitted altogether?

In the opposite direction, it is natural to ask what condition, or combination of conditions, imply that one (or both) of the two entropies $\underline h((f_n))$ or $h((f_n))$ is zero. So far we have no answer. However, we can construct a non-autonomous dynamical system of centred inner functions which has topological entropy zero, which is potentially informative.

The generalisation of topological entropy to non-autonomous dynamical systems has been studied by numerous authors, going back at least to \cite{KolSno}. For the non-autonomous dynamical system of boundary values of sequences of centred inner functions $(f_n)$, with $F_n=f_n\circ\cdots\circ f_1$, the topological entropy can be defined as follows, by analogy with the standard definition for autonomous systems; see \cite[Chapter~4]{Walters}. A subset~$S$ of $\partial\D$ is called $(n, \delta)$\textit{-separated} for~$(f_n)$, where $n\ge 0$ and $\delta > 0$, if any two distinct points $\zeta, \zeta' \in S$ satisfy
\[
\max_{0\le k\le n}|F_k(\zeta)-F_k(\zeta')| \ge \delta.
\]
The \textit{topological entropy} $h_{\rm{top}}(f_n)$ of the sequence $(f_n)$ is then defined as
\[
h_{\rm{top}}((f_n)):= \lim_{\delta \to 0} \limsup_{n\to\infty}\frac1n\log N(n,\delta),
\]
where $N(n,\delta)$ denotes the maximum cardinality of an $(n,\delta)$-separated subset of $\partial\D$ for~$(f_n)$.

Within autonomous dynamics, it is a standard result that topological entropy dominates metric entropy; see \cite[p.~156]{Walters}, for example. In the non-autonomous setting, however, we are only aware of the result of Kawan \cite[Theorem~4.3]{Kawan} that this property holds under the additional assumption that the sequence $(f_n)$ is equicontinuous. However, the estimates developed in the proof of Theorem~C can be adapted to prove the following.
\begin{thmD}
If the sequence $(f_n)$ satisfies the hypotheses of Theorem~C, part~(b), then
\[
h_{\rm{top}}((f_n)) \ge  c\mu,
\]
where~$c$ is a positive constant that depends only on~$a$.
\end{thmD}

Finally, here is our example of a non-autonomous dynamical system of centred inner functions that has topological entropy zero.
\begin{thmE}
Let $(\lambda_n)$ be a sequence in $(0,1)$ such that $\sum_{n=1}^\infty (1-\lambda_n)<\infty$ and put
\[
f_n(z):=z\frac{z+\lambda_n}{1+\lambda_n z}, \quad n \in\N.
\]
Then the forward composition sequence $F_n:=f_n\circ\cdots\circ f_1$ has topological entropy zero on $\partial\D$.
\end{thmE}
The sequence $(f_n)$ is not equicontinuous, since here we have $|f'_n(-1)|=2/(1-\lambda_n) \to \infty$ as $n\to\infty$. It is natural to ask whether the non-autonomous dynamical system in Theorem~E has metric entropy zero.

We mention two aspects of the non-autonomous dynamical system in Theorem~E, which contrast with the fact that this system has topological entropy zero. First, each of the individual maps~$f_n$, when restricted to $\partial \D$, has topological entropy $\log 2$, since it is conjugate there to $z\mapsto z^2$; see~\cite{Shub}. The metric entropy of each $f_n$ is also positive but strictly less than $\log 2$ by \cite[Theorem~4]{Martin}. Note that topological entropy of an autonomous dynamical system is a topological invariant but metric entropy is not in general.

Second, the convergence condition $\sum_{n=1}^\infty (1-\lambda_n)<\infty$ implies that the sequence $F_n=f_n\circ\cdots\circ f_1$ does not tend to the limit~0 in $\D$; see \cite[Theorem~2.1]{BEFRS}. Moreover, by \cite[Theorems~1.1 and~1.2]{Gustavo-inner}, the sequence $(F_n)$ converges locally uniformly in $\D$ to an inner function (even a Blaschke product), $F$ say, such that $F'(0) = \prod_{n=1}^\infty\lambda_n \in (0,1)$. Therefore the limit function~$F$, viewed as an autonomous dynamical system on $\partial\D$, has positive metric entropy by Theorem~\ref{Pomm-ent}, and hence positive topological entropy.

\begin{rem}\label{bigger-lambda}
In each of Theorems~A,~B,~C and~D the condition that $\la_n\geq a>0$ can be omitted by replacing $\lambda_n$ by $\lambda'_n=\frac12(1+\lambda_n)>1/2$ as an upper bound for $|f'_n(0)|$. This change replaces $\lambda_n\cdots\lambda_1$ in the estimates by $\lambda'_n\cdots \lambda'_1$, which is a larger product but one that also tends to~0 whenever $\lambda_n\cdots \lambda_1$ tends to~0. With this larger product in the estimates, the constants~$\delta$ and~$K$ are absolute.
\end{rem}



\section{Preliminary results}
To prove Theorems~A and~B, we need the following simple result on infinite series; see \cite[p.120--121]{inequalities} for related results.

\begin{lem}\label{lem:exp}
Let $0<a_n<1$ for $n \geq 1$ and $c>0$. Then
\[ \sum_{n=1}^{\infty} \frac{a_{n+1}}{\exp(c(a_1+ \dots + a_n))} \leq \frac{e^c}{c}.\]
\end{lem}
\begin{proof}
First,
\begin{eqnarray} \frac{1}{\exp(c(a_1+ \dots + a_n))} - \frac{1}{\exp(c(a_1+ \dots + a_{n+1}))}&=& \frac{\exp(ca_{n+1})-1}{\exp(c(a_1+ \dots + a_{n+1}))} \nonumber \\ &\geq& \frac{ca_{n+1}}{\exp(ca_{n+1}) \exp(c(a_1+ \dots + a_n))}.\nonumber
\end{eqnarray}
Next, telescoping cancellation gives
\[\sum_{n=1}^{\infty}\left(\frac{1}{\exp(c(a_1+ \dots + a_n))} - \frac{1}{\exp(c(a_1+ \dots + a_{n+1}))}\right) \le \frac{1}{e^{ca_1}}.\]
Combining these inequalities, we obtain
\[\frac{1}{e^{ca_1}} \geq \sum_{n=1}^{\infty} \frac{ca_{n+1}}{\exp(ca_{n+1}) \exp(c(a_1+ \dots + a_n))}\ge \frac{c}{e^c}\sum_{n=1}^{\infty} \frac{a_{n+1}}{\exp(c(a_1+ \dots + a_n))},\]
and the result follows.
\end{proof}

We also need L\"owner's lemma, which can be found in \cite[Proposition 4.15]{Pommerenke} and \cite[Theorem 7.1.8 and Proposition 7.1.4 part~(4)]{BracciContrerasDM}, in particular the case of equality for inner functions, which can be found in \cite{pommerenke-ergodic} and also in \cite[Corollary~1.5(b)]{doering-mane}, for example. Here we denote the harmonic measure of a Borel set~$A$ in the boundary of a domain~$U$ by $\omega(z, A, U)$, for $z\in U$.
\begin{lem}[L\"owner's lemma] \label{lem:Lowner}
Let~$f$ be a holomorphic self-map of~$\D$  and let $S\subset \partial \D$ be a Borel set. Then
\begin{equation}\label{eqtn:Ransford Phil}
\omega(z, f^{-1}(S), \D)\leq \omega(f(z), S,\D),\;\text{ for } z\in\D,
\end{equation}
with equality if $f$ is an inner function.
\end{lem}

\section{Proof of Theorem~B}
We prove a somewhat more general version of Theorem~B in which the exponent in the denominator can be arbitrarily close to $2/a$ at the expense of an even larger constant~$K$.
\begin{thm}\label{thm:B2}
For $n\ge 1$, let $f_n$ be centred inner functions with $|f_n'(0)|\le \lambda_n$, where \\$0<a\le \lambda_n<1$ and $\lambda_n \cdots \lambda_1 \to 0$ as $n \to \infty$. Also, let $\eps\in (0,1]$.

Then $F_n= f_n \circ \cdots \circ f_1$ satisfies
\begin{equation}\label{eq:est2a}
|F_n(z)| \leq K \lambda_n \cdots \lambda_1 \frac{|z|}{(1-|z|)^{p}} ,\quad \text{for} \; z \in \mathbb{D},
\end{equation}
where $p=2/a+\eps$ and $K=K(a,\eps)=\exp(48e^{1/2}/(a^2\eps))$, for $0<a < 1$.
\end{thm}

The proof of Theorem~\ref{thm:B2} is based on the auxiliary functions
\begin{equation}\label{eq:psin}
\psi_n(z):=z \frac{z+\lambda_n}{1+\lambda_n z}\quad\text{and}\quad \Psi_n:= \psi_n \circ \cdots \circ \psi_1,
\end{equation}
with $\Psi_0(z)=z$. Since $|f'_n(0)|\le \lambda_n$, we have
\[
|f_n(z)| \leq|z| \frac{|z|+\lambda_n}{1+\lambda_n|z|}= \psi_n(|z|),\quad \text{for }z\in\D, n\ge 1;
\]
see, for example, \cite[p.217]{BC92} or \cite[Corollary~2.4]{BEFRS}. As each $\psi_n$ is increasing on $[0,1]$, we have
\begin{equation}\label{eq:fn-Psin}
|F_n(z)| \leq (\psi_n \circ \cdots \circ \psi_1)(|z|)=\Psi_n(|z|),\quad\text{for } z\in\D.
\end{equation}

We prove Theorem~\ref{thm:B2} by estimating $\Psi_n(|z|)$ in two stages, first obtaining two asymptotically weaker inequalities.
\begin{lem}\label{lem:weakineq}
For $n\ge 1$, let $\psi_n$ and $\Psi_n$ be defined by \eqref{eq:psin}, where $0<a\le \lambda_n<1$ and $\lambda_n \cdots \lambda_1 \to 0$ as $n \to \infty$. Put $\mu_n=1-\lambda_n$ for $n\in\N$. Then, for $z\in\D$ and $n\ge 1$,
\begin{equation}\label{eq:befrs}
\Psi_n(|z|)\le |z| \prod_{k=1}^n(1-c(z)\mu_k),
\end{equation}
where $c(z)=\frac12(1-|z|)$ and
\begin{equation}\label{eq:est1}
\Psi_n(|z|) \leq \lambda_n \cdots \lambda_1 |z| \exp \left(\frac{4e^{1/2}/a}{1-|z|}\right).
\end{equation}
\end{lem}

\begin{proof}[Proof of Lemma~\ref{lem:weakineq}]
The estimate \eqref{eq:befrs} was given in the proof of \cite[Theorem~2.5]{BEFRS} using the following argument: for $z\in\D$,
\begin{align}\label{befrs-arg}
\Psi_n(|z|) & = \psi_n(\Psi_{n-1}(|z|))\notag\\
&=\Psi_{n-1}(|z|)\left(\frac{\Psi_{n-1}(|z|)+\lambda_{n}}{1+\lambda_{n}\Psi_{n-1}(|z|)}\right)\notag\\
& = \Psi_{n-1}(|z|)\left(1-\frac{1-\Psi_{n-1}(|z|)}{1+\lambda_n \Psi_{n-1}(|z|)}\,\mu_n\right)\notag\\
&\le \Psi_{n-1}(|z|)(1-c(z)\mu_n),
\end{align}
where $c(z)=\tfrac12 (1-|z|)$. The final inequality holds because $\lambda_{n} \Psi_{n-1}(|z|)<1$ and because $\Psi_{n-1}(|z|) \leq |z|$ by Schwarz's lemma.

Therefore, since $\Psi_0(z)=z$, we have
\[
\Psi_n(|z|)\le |z| \prod_{k=1}^n(1-c(z)\mu_k),\quad\text{for } z\in\D,
\]
which is \eqref{eq:befrs}.

Since $e^{-x} \geq 1-x,$ for $0 \leq x \leq 1$, we obtain from \eqref{eq:befrs} that
\begin{equation}\label{ineq:Psi-n}
|\Psi_n(z)| \leq |z| \exp(-c(z)(\mu_1+ \dots +\mu_n)),\quad\text{for } z\in\D, n\ge 1.
\end{equation}
Observe that
\[
\sum_{n=1}^\infty \mu_n = \infty,
\]
since $\lambda_n \cdots \lambda_1 \to 0$ as $n \to \infty$ and so, for $z \in \mathbb{D},$ we have
\[
\Psi_n(z) \to 0 \; \text{ as}\;n \to \infty.
\]
We are seeking a bound for $\Psi_n(|z|)$ involving the product $\lambda_n\cdots \lambda_1=(1-\mu_n)\cdots(1-\mu_1)$ rather than the product $(1-c(z)\mu_n)\cdots(1-c(z)\mu_1)$ in \eqref{befrs-arg}. To obtain this, we replace the rearrangement and estimate in \eqref{befrs-arg} by the following
\begin{align}\label{eq:new-arg}
\Psi_n(|z|) & = \psi_n(\Psi_{n-1}(|z|))\notag\\
&=\Psi_{n-1}(|z|)\left(\frac{\Psi_{n-1}(|z|)+\lambda_{n}}{1+\lambda_{n}\Psi_{n-1}(|z|)}\right)\notag\\
& =  \lambda_n\Psi_{n-1}(|z|)\left(\frac{1+\Psi_{n-1}(|z|)/\lambda_n}{1+\lambda_n \Psi_{n-1}(|z|)}\right)\notag\\
& \le \lambda_n\Psi_{n-1}(|z|)\left(1+(1/\lambda_n-\lambda_n)\Psi_{n-1}(|z|)\right)\quad\text{(since }0<\lambda_n<1)\notag\\
& = \lambda_n\Psi_{n-1}(|z|)\left(1+\frac{1-\lambda_{n}^2}{\lambda_{n}}\Psi_{n-1}(|z|)\right).
\end{align}

Applying \eqref{eq:new-arg} repeatedly, substituting \eqref{ineq:Psi-n}, and using Lemma~\ref{lem:exp} together with the hypothesis that $a\le\lambda_n <1$, for $n\ge 1$, we obtain
\begin{eqnarray}
\Psi_n(|z|) &\leq& \lambda_n\cdots \lambda_1|z| \prod_{k=0}^{n-1} \left(1+ \frac{1-\lambda_{k+1}^2}{\lambda_{k+1}}\Psi_k(|z|)\right) \nonumber \\
&\leq& \lambda_n\cdots \lambda_1|z| \prod_{k=0}^{n-1}\left(1+ \frac{(2/a)\mu_{k+1}|z|}{\exp(c(z)(\mu_1+\dots+\mu_k)}\right)\nonumber \\
&\leq& \lambda_n\cdots \lambda_1|z| \exp \left(\sum_{k=0}^{n-1} \frac{(2/a)\mu_{k+1}}{\exp(c(z)(\mu_1+\dots+\mu_k)}  \right) \nonumber \\
&\leq &\lambda_n\cdots \lambda_1|z| \exp\left(\frac{2e^{c(z)}/a}{c(z)}\right)\nonumber\\
& \le& \lambda_n\cdots \lambda_1 |z| \exp \left(\frac{4e^{1/2}/a}{1-|z|}\right),\nonumber
\end{eqnarray}
since $c(z)= \frac{1}{2}(1-|z|)\le 1/2$. This proves~\eqref{eq:est1}.
\end{proof}
We will use the estimate \eqref{eq:est1} to prove Theorem~\ref{thm:B2} by applying \eqref{eq:est1} on a certain closed disc in $\D$ to give a bound for $\Psi_n$ on that disc of the required form, and we then extend this bound to the whole of~$\D$ by a process of backward iteration.
\begin{proof}[Proof of Theorem~\ref{thm:B2}]
To prove the estimate~\eqref{eq:est2a}, where $p=p(a,\eps)=2/a+\eps$ is the exponent in the denominator, we choose $b=b(a,\eps)$ such that
\begin{equation}\label{eq:bchoice}
0<4(1-b)<1-\frac{2}{ap}=\frac{a\eps}{2+a\eps};
\end{equation}
for example, we can take $b=1-a\eps/12$, since $a \in (0,1)$ and $\eps\in (0,1]$. The reason for these choices of~$p$ and~$b$ will become clear later in the proof.

Next we observe that
\[
\exp \left(\frac{4e^{1/2}/a}{1-|z|}\right) \le K \frac{1}{(1-|z|)^p},\quad\text{for } |z| \leq b=1-a\eps/12,
\]
where
\[
K=K(a,\eps)=\exp \left(\frac{48e^{1/2}}{a^2\eps}\right).
\]
This holds since this value of~$K$ satisfies $\log K+\log(1/x^p) \ge (4e^{1/2}/a)/x$ for $a\eps/12 \le x\le 1$. Therefore, by \eqref{eq:est1},
\begin{equation}\label{eq:bound-inside-b}
\Psi_n(|z|) \leq K\lambda_n \cdots \lambda_1 \frac{|z|}{(1-|z|)^p}, \quad\text{for } |z| \leq b.
\end{equation}

Now let $\Psi_{n,m}:= \psi_n \circ \cdots \circ \psi_{m+1}$, for $n > m \geq 0$. By relabelling the $\psi_n$ sequence, the estimate \eqref{eq:est1} shows that, for $n>m\ge 0$, with the constant $K=K(a,\eps)$ defined as above,
\begin{align}\label{eq:uniform constant}
\Psi_{n,m}(|z|) &\leq \lambda_n \cdots \lambda_{m+1}|z| \exp \left(\frac{4e^{1/2}/a}{1-|z|}\right),\quad\text{for }z\in\D,\notag\\
&\le K \lambda_n \cdots \lambda_{m+1} \frac{|z|}{(1-|z|)^p},\quad\text{for } |z| \leq b.
\end{align}

To extend the bound \eqref{eq:bound-inside-b} to the whole of $\D$, we take any $r_1 \in (b,1)$ and define the sequence $(r_k)$ by $r_{k+1}=\psi_k(r_k)$, $k=1,2,\dots$. Then $(r_k)$ is decreasing with $r_k \to 0$ as $k \to \infty$, by \eqref{eq:befrs}, since $\lambda_n \cdots \lambda_1 \to 0$ as $n\to\infty$. Therefore, we can choose~$m\in\N$ so that $r_{m+1} \leq b < r_m$ and by applying \eqref{eq:uniform constant} we obtain, for $n>m\ge 0$, that
\[
\Psi_{n,m}(r_{m+1}) \leq K\lambda_n \cdots \lambda_{m+1} \frac{r_{m+1}}{(1-r_{m+1})^p},
\]
which can be rewritten as
\begin{equation}\label{eq:r_m1}
\Psi_{n,m-1}(r_{m}) \leq K\lambda_n \cdots \lambda_{m+1} \frac{\psi_m(r_m)}{(1-\psi_m(r_{m}))^p}.
\end{equation}
We wish to deduce from \eqref{eq:r_m1} that, for $n>m\ge 0$,
\begin{equation}\label{eq:r_m2}
\Psi_{n,m-1}(r_{m}) \leq K\lambda_n \cdots \lambda_{m} \frac{r_m}{(1-r_{m})^p},
\end{equation}
with the same constant~$K$. It suffices to show that
\begin{equation}\label{eq:suf}
\frac{\psi_m(r_m)}{(1-\psi_m(r_{m}))^p} \le \lambda_{m} \frac{r_m}{(1-r_{m})^p},\quad\text{that is,} \quad\frac{r_{m+1}}{\lambda_mr_m} \le \left(\frac{1-r_{m+1}}{1-r_m}\right)^p.
\end{equation}

To do this we use the Taylor expansion of $\psi_m$ about 1. We have $\psi_m(1)=1$,
\[ \psi_m'(z) = \frac{\lambda_m z^2+2z+\lambda_m}{(1+\lambda_mz)^2}, \quad \text{so }\; \psi_m'(1)=\frac{2}{1+ \lambda_m},\]
and
\[\psi_m''(z)= \frac{2(1-\lambda_m^2)}{(1+\lambda_m z)^3}.\]
By Taylor's theorem,
\[r_{m+1}=\psi_m(r_m)= \psi_m(1)+ \psi_m'(1)(r_m-1)+ \frac{1}{2!}\psi_m''(\xi_m)(r_m-1)^2,\]
for some $\xi_m \in (r_m,1)$. Thus
\[
1-r_{m+1}= \frac{2}{1+\lambda_m}(1-r_m) - \frac{1-\lambda_m^2}{(1+\lambda_m \xi_m)^3} (1-r_m)^2;
\]
that is,
\[
\frac{1-r_{m+1}}{1-r_m}=   \frac{2}{1+\lambda_m}- \frac{1-\lambda_m^2}{(1+\lambda_m \xi_m)^3} (1-r_m).
\]
Since $b<r_m<\xi_m<1$ and $0<\lambda_m <1$, we deduce that
\begin{align}\label{eq:lambda1}
\frac{1-r_{m+1}}{1-r_m}&= 1 + \left( \frac{1}{1+\lambda_m}- \frac{(1-r_m)(1+\lambda_m)}{(1+\lambda_m \xi_m)^3}\right)(1-\lambda_m)\notag\\
&> 1+\left(\frac12-2(1-b)\right)(1-\lambda_m).
\end{align}

On the other hand, since $r_{m+1}<r_m$ and $\lambda_m\ge a$, we have
\begin{equation}\label{eq:lambda2}
\frac{r_{m+1}}{\lambda_m r_m}< \frac{1}{\lambda_m}=1+ \frac{1-\lambda_m}{\lambda_m}\le 1 + \frac1a(1-\lambda_m).
\end{equation}
Now, the choice of $b=b(a,\eps)$ in \eqref{eq:bchoice} was made so that
\[
4(1-b)<1-\frac{2}{ap},\quad\text{that is,}\quad\frac1a < p\left(\frac12-2(1-b)\right).
\]
Hence, by \eqref{eq:lambda1} and \eqref{eq:lambda2}, and the fact that $p>1$, we have
\begin{align*}
\frac{r_{m+1}}{\lambda_m r_m} &\le 1 + \frac1a(1-\lambda_m)\\
&< 1 + p\left(\frac12-2(1-b)\right)(1-\lambda_m)\\
&< \left( 1 + \left(\frac12-2(1-b)\right)(1-\lambda_m) \right)^p\\
&< \left( \frac{1-r_{m+1}}{1-r_m}\right)^p,
\end{align*}
which is (\ref{eq:suf}).

Thus we have deduced (\ref{eq:r_m2}) from (\ref{eq:r_m1}). Since we have $b<r_k<1$, $r_{k+1}<r_k$ and $a\le \lambda_k<1$, for $k=1, \ldots, m-1$, we can repeat this process~$m-1$ times to obtain
\[
\Psi_{n,0}(r_1)= \psi_n \circ \cdots \circ \psi_1 (r_1) \leq K\lambda_n \cdots \lambda_1 \frac{r_1}{(1-r_1)^p}.
\]
Thus we deduce by \eqref{eq:fn-Psin} that, for $|z|=r_1,$
\[
|F_n(z)| \leq \Psi_{n,0}(r_1) \leq K\lambda_n \cdots \lambda_1 \frac{|z|}{(1-|z|)^p},
\]
where $K=K(a,\eps)$. Since $r_1\in (b,1)$ was arbitrary, the proof of Theorem~\ref{thm:B2} is complete.
\end{proof}
\begin{rem*}
It follows from a result in \cite[Section~3]{PommComps} that under the hypotheses of Theorem~\ref{thm:B2} the functions
\[
G_n(z):=\frac{(\psi_n\circ\cdots\circ \psi_1)(z)}{\lambda_n\cdots \lambda_1}=\frac{\Psi_n(z)}{\lambda_n\cdots \lambda_1},
\]
where $\psi_n$ and $\Psi_n$ are defined by \eqref{eq:psin}, converge locally uniformly in~$\D$ to a limit function of the form $G(z)=z+\cdots$, which is univalent in the disc $\{z:|z|<a/(1+\sqrt{1-a^2}\,)\}$. Using the Koebe distortion theorem (see \cite[Theorem~1.3]{Pommerenke}, for example), a uniform bound on $|G(z)|$ can then be obtained in a smaller disc. This bound can then be used to give an upper bound for $|G_n(z)|$ and therefore for $|F_n(z)|$ of the required form, for~$z$ in that smaller disc. However, for values of~$a$ that are not close to~1, it seems problematic to transfer this upper bound on $|F_n(z)|$ from that smaller disc to the whole of~$\D$ by backward iteration, as we did in the proof of Theorem~\ref{thm:B2} using the bound given by \eqref{eq:est1}, applied on a disc that is close to the size of $\D$.

On the other hand, for values of~$a$ that are close to~1, the fact that the limit function~$G$ defined above is univalent on a disc that approximates~$\D$ suggests that the exponent $2/a+\eps$ in the denominator of \eqref{eq:est2a} is close to being best possible.
\end{rem*}


\section{Proof of Theorem~A}
Theorem~A states that if $f_n$ are inner functions with $f_n(0)=0$ and $|f_n'(0)|\le \lambda_n$, where $0<a\le \lambda_n <1$ and $\lambda_n \cdots \lambda_1 \to 0$ as $ n \to \infty$, then $F_n= f_n \circ \cdots \circ f_1$ satisfies
\begin{equation}\label{eq:newmixing1}
\left|\frac{|A \cap F_n^{-1}(E)|}{|E|} - \frac{|A|}{2\pi}\right| \le K(\lambda_n \cdots \lambda_1)^{\delta},
\end{equation}
for all arcs $A\subset \partial\D$ and measurable sets $E$ in $\partial \D$ with $|E|>0$, where $\delta=(6(1/a+1))^{-1}$ and $K=K(a)$ is a positive constant. Note that $\delta=1/(3(p+1))$, where~$p=p(a)=2/a+1$ is the exponent of the denominator in Theorem~B.

Here we follow the ingenious argument that Pommerenke used to deduce Theorem~\ref{thm:Pom3} from Lemma~\ref{lem:Pom3} though, because of the increased complexity, we give more detail at some points.

\begin{proof}[Proof of Theorem~A]
For $n\in \N$, we let
\[
r_n= 1- (\lambda_n \cdots \lambda_1)^{3\delta}\quad\text{and} \quad \varepsilon_n= (1-r_n)^{1/3}=(\lambda_n \cdots \lambda_1)^{\delta}.
\]
Then the estimate \eqref{eq:est2} from Theorem~B gives a positive constant~$K=K(a)$ such that, for $|z|=r_n$ and $n\ge 1$,
\begin{align}\label{Fnest}
|F_n(z)| &\leq K\lambda_n \cdots \lambda_1 |z|\frac{1}{(1-|z|)^p}\notag\\
&= \frac{K\lambda_n \cdots \lambda_1}{(\lambda_n \cdots \lambda_1)^{3\delta p}}\notag\\
&= K(\lambda_n \cdots \lambda_1)^{1/(p+1)}\notag\\
&= K(\lambda_n \cdots \lambda_1)^{3\delta} \to 0 \quad\text{as } n\to \infty,
\end{align}
since $\delta=1/(3(p+1))$.

The idea behind the proof is as follows. By \eqref{Fnest} and the case of equality in L\"owner's lemma, Lemma~\ref{lem:Lowner}, we deduce that for $|z|=r_n$ we have
\[
\omega(z, F_n^{-1}(E),\D) = \omega(F_n(z), E,\D) \simeq \omega(0, E, \D)=\frac{|E|}{2\pi}.
\]
Therefore $\omega(z, F_n^{-1}(E),\D)$ is essentially constant for $|z|=r_n$, so it is plausible that the set $F_n^{-1}(E)$ is distributed uniformly around $\partial\D$, with
\[
|F_n^{-1}(E)\cap A|\simeq  |F_n^{-1}(E)|\frac{|A|}{2\pi} = |E|\frac{|A|}{2\pi},
\]
for any arc $A\subset \partial\D$.

To make this idea precise, we need estimates for the Poisson kernel of $\D$,
\[
P(z,\zeta)=\frac{1-|z|^2}{|\zeta-z|^2},\quad z\in\D, \zeta\in \partial \D.
\]
On taking $z=r_ne^{i\phi}$, $\zeta=e^{i\theta}$ and $|\phi-\theta|\ge \epsilon_n$, we obtain:
\begin{equation}\label{eq:Pomm2.3}
P(r_ne^{i\phi},e^{i\theta})\le \frac{1-r_n^2}{|e^{i\varepsilon_n}-r_n|^2} \le \frac{2(1-r_n)}{\sin^2 \epsilon_n}=\frac{2\epsilon_n^3}{\sin^2 \epsilon_n}\le 3\epsilon_n, \quad\text{for } n\in\N,
\end{equation}
since $\lambda_n \cdots \lambda_1<1$ and $\sin \epsilon_n \ge 2\epsilon_n/\pi$.

Let $A$ be any arc of $\partial\D$ and, for $n\ge 1$, put
\[
A_n^+=\{e^{i\theta}:\min\{|\theta-\phi|:e^{i\phi}\in A\}\le \epsilon_n\},
\]
and
\[
A_n^-=\{e^{i\theta}\in A:\min\{|\theta-\phi|:e^{i\phi}\in \partial \D\setminus A\}\ge \epsilon_n\};
\]
if necessary take $A_n^+=\partial\D$ or $A_n^-=\emptyset$. Then consider the harmonic measures
\begin{equation}\label{eq:harm-meas}
\omega(z,A_n^\pm,\D)=\frac{1}{2\pi}\int_{A_n^\pm}P(z,\zeta)\,|d\zeta|.
\end{equation}
By \eqref{eq:Pomm2.3} and the definitions of $A_n^\pm$, we have
\begin{equation}\label{eq:Pomm2.5}
\omega(r_n\zeta,A_n^+,\D)\ge 1-3\varepsilon_n,\;\;\text{for }\zeta\in A\quad\text{and}\quad \omega(r_n\zeta,A_n^-,\D)\le 3\eps_n,\;\;\text{for }\zeta\in \partial\D \setminus A.
\end{equation}

Now recall that $F_n$, $n\ge 1,$ are inner functions such that $F_n(0)=0$ and put $E_n := F_n^{-1}(E)$. Then by Lemma~\ref{lem:Lowner} we have $|E_n|=|E|$ and $\omega(r_nz,E_n,\D)=\omega(F_n(r_n z),E,\D)$, so the following identity \cite[(2.9)]{pommerenke-ergodic} follows from \eqref{eq:harm-meas}:
\begin{align}\label{eq:identity}
\int_{E_n} \omega(r_n\zeta,A_n^{\pm},\D)\, |d\zeta| &= \int_{E_n}\left(\frac{1}{2\pi}\int_{A_n^\pm}\frac{1-r_n^2}{|z-r_n\zeta|^2}\,|dz|\right)|d\zeta|\notag\\
&= \int_{A_n^\pm}\left(\frac{1}{2\pi}\int_{E_n}\frac{1-r_n^2}{|\zeta-r_n z|^2}\,|d\zeta|\right)|dz|\notag\\
&= \int_{A_n^{\pm}} \left(\frac{1}{2\pi} \int_E \frac{1-|F_n(r_n z)|^2}{|\zeta -F_n(r_n z)|^2}\, |d\zeta| \right) |dz|.
\end{align}
From now on we assume~$n$ is so large that
\begin{equation}\label{eq:nlarge}
K(\lambda_n \cdots \lambda_1)^{3\delta}\le 1/2 \quad\text{and} \quad \eps_n= (\lambda_n \cdots \lambda_1)^{\delta}\le 1/6.
\end{equation}
The estimate \eqref{eq:newmixing1} automatically holds if either of the inequalities in \eqref{eq:nlarge} fails since the left-hand side of \eqref{eq:newmixing1} is at most~1.

Then
\[
|F_n(r_nz)|\le K(\lambda_n \cdots \lambda_1)^{3\delta}\le 1/2,\quad \text{for } |z|=1,
\]
by \eqref{Fnest}. Since
\[
\left|\frac{1-|z|^2}{|\zeta-z|^2}-1\right|\le 4|z|,\quad\text{for }|z|\le 1/2,
\]
we deduce that, for such~$n$ and any $|z|=1$,
\begin{align}
\left|\int_E \frac{1-|F_n(r_n z)|^2}{|\zeta -F_n(r_n z)|^2}\, |d\zeta| - |E| \right| &\leq 4|F_n(r_nz)|.|E|\notag\\
& \leq 4K(\lambda_n \cdots \lambda_1)^{3\delta}|E|.
\end{align}
Hence, by \eqref{eq:identity},
\begin{align}\label{eq:keyineq}
\left| \int_{E_n} \omega(r_n \zeta, A_n^{\pm}, \D)\,|d\zeta|- \frac{|A_n^\pm|}{2\pi}|E| \right|\notag
&= \left| \int_{A_n^{\pm}}\left(\frac{1}{2\pi} \int_E \frac{1-|F_n(r_n z)|^2}{|\zeta -F_n(r_n z)|^2} \,|d\zeta|- \frac{|E|}{2\pi}\right) |dz| \right|\notag\\
& \le K (\lambda_n \cdots \lambda_1)^{3\delta}|E|.|A_n^\pm|.
\end{align}
We deduce from \eqref{eq:keyineq} and \eqref{eq:Pomm2.5} that, for such~$n$,
\begin{align*}
(1-3\epsilon_n)|A\cap E_n| &\le \int_{A\cap E_n} \omega(r_n \zeta, A_n^+, \D)\,|d\zeta|\\
&\le \int_{E_n} \omega(r_n \zeta, A_n^+, \D)\,|d\zeta|\\
&\le \left(\frac{|A| + 2\varepsilon_n}{2\pi}\right)|E|\left(1+ K(\lambda_n \cdots \lambda_1)^{3\delta}\right)\\
&= \left(\frac{|A|}{2\pi} + \frac{\eps_n}{\pi}\right)|E|\left(1+ K\eps_n^3\right).
\end{align*}
Therefore, since $\eps_n= (\lambda_n \cdots \lambda_1)^{\delta}\le 1/6$, we obtain
\begin{align*}
\frac{|A\cap E_n|}{|E|}&\le (1+6\eps_n)\left(\frac{|A|}{2\pi} + \frac{\eps_n}{\pi}\right)\left(1+ K\eps_n^3\right)\\
&\le \frac{|A|}{2\pi}+(K/12+7)\epsilon_n\\
&= \frac{|A|}{2\pi}+ K_1\eps_n,
\end{align*}
where $K_1=K_1(a)$ is a positive constant. On the other hand, by \eqref{eq:Pomm2.5} again,
\begin{align*}
\int_{E_n}\omega(r_n \zeta, A_n^-, \D)|d\zeta| &= \int_{A\cap E_n}\omega(r_n \zeta, A_n^-, \D)|d\zeta| + \int_{E_n\setminus A}\omega(r_n \zeta, A_n^-, \D)|d\zeta|\\
&\le |A\cap E_n|+3\epsilon_n |E_n|\\
&= |A\cap E_n|+3(\lambda_n \cdots \lambda_1)^{\delta}|E|.
\end{align*}
Hence, by \eqref{eq:keyineq} again,
\begin{align*}
|A\cap E_n|&\ge \int_{E_n}\omega(r_n \zeta, A_n^-, \D)|d\zeta|-3\eps_n|E|\\
&\ge \frac{|A_n^-|}{2\pi} |E| -K\eps_n^3|E|.|A_n^-|-3\eps_n|E|\\
&= \frac{|A|-2\eps_n}{2\pi}-K\eps_n^3|E|.|A_n^-|-3\eps_n|E|\\
&= \frac{|A|}{2\pi}-\frac{\eps_n}{\pi}-K\eps_n^3|E|.|A_n^-|-3\eps_n|E|\\
&\ge \frac{|A|}{2\pi} - K_2\eps_n,
\end{align*}
where $K_2=K_2(a)$ is a positive constant. Hence
\[
\left| \frac{|A \cap E_n|}{|E|} - \frac{|A|}{2\pi} \right| \leq \max\{K_1,K_2\}\eps_n= \max\{K_1,K_2\}(\lambda_n \cdots \lambda_1)^{\delta}.
\]
This completes the proof of Theorem~A.
\end{proof}


\section{Entropy: Proofs of Theorems~C and~D}

Let us assume that the sequence of inner functions $(f_n)$ satisfies the conditions in Theorems~C and~D; that is, $f_n$ are centred inner functions with $\la_n=\vert f_n'(0)\vert \geq a>0$ and $\lambda_n \cdots \lambda_1 \to 0$ as $n \to \infty$, and put $\mu_n:=1-\la_n$ for n=1,2,\ldots. Then the sequence $(f_n)$ satisfies the hypotheses of Theorem~A, and we let $K=K(a)$ and $\delta=\delta(a)$ be the positive constants in Theorem~A. Also, put $F_n=f_n \circ \cdots \circ f_1$, for $n\ge 1$.

We shall first prove Theorem~C, part~(b). Although we shall not use it until the end of the proof, observe at this point that the hypothesis of Theorem~C, part~(b) implies that
\begin{equation} \label{eq:key}
 \sum_{k=1}^n \mu_k  \geq \tfrac12\mu n, \text{\ \ \ \ \ for infinitely many~$n$.}
\end{equation}

We now group the inner functions in blocks. Since the sum of the $\mu_k$'s is divergent and $0<\mu_k<1$ for all $k$, we can choose a sequence $(n_j)_{j\geq 0}$ such that  $n_0=0$ and
\begin{equation}\label{eq:blocks}
\mu_{n_{j-1}+1}+ \mu_{n_{j-1}+2} +\cdots+ \mu_{n_j} \in \left[\frac{\log 4K}{\delta}, \frac{\log 4K}{\delta}+1\right],\quad\text{for }j=1,2,\ldots.
\end{equation}
Then
\[
\delta  \log (\la_{{n_{j-1}+1}} \cdots \la_{n_j}) =\delta \log(1-\mu_{n_{j-1}+1})\cdots (1-\mu_{n_j})
\leq -\delta (\mu_{n_{j-1}+1}+\cdots+ \mu_{n_j}) \leq \log \frac{1}{4K};
\]
that is,
\begin{equation}\label{eq:14}
K \left(\la_{n_j} \cdots \la_{{n_{j-1}+1}}\right)^\delta \leq \frac14, \text{\ \ \ for $j=1,2,\ldots$.}
\end{equation}

We now define a new sequence $(g_j)$ corresponding to these blocks, as follows:
\[
g_j:=f_{n_j}\circ \cdots\circ f_{n_{j-1}+1},\quad j=1,2,\ldots;
\]
for example, $g_1=f_{n_1}\circ \cdots\circ f_1=F_{n_1}$.

We start with Theorem~C, part~(b). To show that the metric entropy is positive, it suffices to find one partition of $\partial\D$ for which
\[
\limsup_{n\to\infty}\frac{1}{n+1} \sum_{k_0=1}^N \ldots \sum_{k_n=1}^N
\vert A_{k_0\ldots k_n}\vert \log \frac{2\pi}{\vert A_{k_0\ldots k_n}\vert}>0,
\]
where $A_{k_0\ldots k_n}$ is defined by \eqref{eq:Ak0kn}. To this end, let us consider the partition defined by $A_1=\{z\in\partial \D : {\rm Im}(z)\geq0\}$ and $A_2=\{z\in\partial \D : {\rm Im}(z)<0\}$.  It follows by \eqref{eq:14} and Theorem~A (applied to the composition $g_j$) that, for $j\ge 1$, $k=1,2$, and for all measurable sets $E\subset \partial \D$ with $|E|>0$,
\begin{equation}\label{eq:mixtopent}
\left|\frac{|A_k \cap g_j^{-1}(E)|}{|E|} - \frac{|A_k|}{2\pi}\right| \le K(\la_{n_j} \cdots \la_{{n_{j-1}+1}})^{\delta} \leq \frac14,
\end{equation}
which, since $|A_k|=\pi$ for $k=1,2$, implies that
\begin{equation} \label{eq:ind}
|A_k \cap g_j^{-1}(E)| \leq \frac34 |E|,
\end{equation}
for all $j\geq 1$, $k=1,2$ and all measurable sets $E\subset\partial \D$ with $|E|>0$.

We now prove the following claim.
\begin{claim} \label{claim}
 For any $m\geq 0$ and $n\geq n_m$,
\[
\vert A_{k_0\ldots k_n}\vert \leq \pi \left(\frac34\right)^m,
\]
where $k_n\in\{1,2\}$ for $n=1,2,\ldots.$
\end{claim}
\begin{proof}
Observe that, for any $m\geq 0$ and $n\geq n_m$,
\[
A_{k_0\ldots k_n} \subset A_{k_0\ldots k_{n_m}},
\]
and hence, for such $n$,
\begin{equation}\label{eq:ineq1}
\vert A_{k_0\ldots k_n}\vert \leq \vert A_{k_0\ldots k_{n_m}}\vert.
\end{equation}
Recall that,
\[
F_{n_m}=f_{n_m}\circ\cdots\circ f_2\circ f_1 =g_m\circ \cdots \circ g_1,\quad \text{for $m\geq 1$.}
\]
Then, since the itineraries for the blocks $(g_j)$ are less restrictive than those for $(f_n)$, we have
\[
\vert A_{k_0\ldots k_{n}}\vert = \vert A_{k_0}\cap F_1^{-1}(A_{k_1})\cap \cdots\cap F_n^{-1}(A_{k_n})\vert
\leq \vert A_{k_0}\cap F_{n_1}^{-1}(A_{k_{n_1}})\cap \cdots\cap F_{n_m}^{-1}(A_{k_{n_m}}) \vert.
\]
Now, to estimate the latter, let us write, using (\ref{eq:ind}),
\[
 \vert A_{k_0}\cap F_{n_1}^{-1}(A_{k_{n_1}})\cap \cdots\cap F_{n_m}^{-1}(A_{k_{n_m}})\vert
  = \vert A_{k_0} \cap g_1^{-1}(E_1)\vert \leq \frac34 \vert E_1\vert,
 \]
 where
 \[
 E_1= A_{k_{n_1}} \cap g_2^{-1}(A_{k_{n_2}}) \cap \cdots\cap (g_m\circ\cdots\circ g_2)^{-1}(A_{k_{n_m}}).
 \]
Proceeding inductively, we define sets $E_j=A_{k_{n_j}} \cap g_{j+1}^{-1}(E_{j+1})$, for $j=2,\ldots,m-1,$ and $E_m=A_{k_{n_m}}$.  By applying (\ref{eq:ind}) repeatedly we obtain
\[
\vert A_{k_0\ldots k_{n_m}}\vert=  \vert A_{k_0}\cap F_{n_1}^{-1}(A_{k_{n_1}})\cap \cdots\cap F_{n_m}^{-1}(A_{k_{n_m}})\vert \leq \left(\frac34\right)^m \vert E_m| =\pi \left(\frac34\right)^m.
 \]
The claim then follows from (\ref{eq:ineq1}).
\end{proof}

To prove the final estimate, we consider values of~$n$ such that \eqref{eq:key} holds and take $n_m$ such that $n_m\leq n< n_{m+1}$. From the Claim we obtain
\begin{equation} \label{eq:final}
\frac{1}{n+1} \sum_{k_0=1}^2 \ldots \sum_{k_n=1}^2 \vert A_{k_0\ldots k_n}\vert \log \frac{2\pi}{\vert A_{k_0\ldots k_n}\vert}
\geq \frac{1}{n_{m+1}+1}  \sum_{k_0=1}^2 \ldots \sum_{k_n=1}^2 \vert A_{k_0\ldots k_n}\vert
\left( \log \left(\frac43\right)^m +\log 2   \right).
\end{equation}
Now, the choice of~$n$, together with (\ref{eq:key}) and (\ref{eq:blocks}) imply that
\begin{equation}\label{eq:mnmmu}
\tfrac12\mu n_m \le\tfrac12\mu n\le \sum_{k=1}^{n}  \mu_k\leq \sum_{k=1}^{n_{m+1}}  \mu_k \leq (m+1) \left( \frac{\log 4K}{\delta} +1\right).
\end{equation}
So, using the fact that $ \sum_{k_0=1}^2 \ldots \sum_{k_n=1}^2
\vert A_{k_0\ldots k_n}\vert =2\pi$, the equations (\ref{eq:final}) and \eqref{eq:mnmmu} yield the following estimate. Here, after the first line, the limsups are taken along sequences of values of~$n$ and~$m$ where $n$ satisfies \eqref{eq:key} and $n_m\leq n < n_{m+1}$:
\begin{align*}
 h((f_n)) & =  \limsup_{n\to\infty} \frac{1}{n+1} \sum_{k_0=1}^2 \ldots \sum_{k_n=1}^2
\vert A_{k_0\ldots k_n}\vert \log \frac{2\pi}{\vert A_{k_0\ldots k_n}\vert}
 \\
& \ge \limsup_{m\to\infty} \frac{m}{n_{m+1}+1} \left(\sum_{k_0=1}^2 \ldots \sum_{k_n=1}^2
\vert A_{k_0\ldots k_n}\vert\right)  \log \frac43 \\
 & \geq \limsup_{m\to\infty} \frac{2\pi m}{\frac{m+2}{\mu/2}  \left( \frac{\log 4K}{\delta} +1\right) + 1} \log \frac43 \\
 & = \left(\frac{\pi\, \log(4/3) }{\frac{\log 4K}{\delta} +1}\right) \mu >0.
\end{align*}
This concludes the proof of Theorem C, part~(b), with $c=(\pi\, \log(4/3))/(\frac{\log 4K}{\delta} +1)$. Recall that $K=K(a)$ and $\delta=\delta(a)$ are the positive constants in Theorem~A.

The proof of part~(a) is similar, with \eqref{eq:key} replaced by
\begin{equation}\label{eq:keya}
\sum_{k=1}^n \mu_k  \geq \tfrac12\underline\mu n, \text{ for sufficiently large } n.
\end{equation}

The proof of Theorem~D builds on techniques developed in the proof of Theorem~C and is quite short. First, we use the same subsequence $(\mu_{n_j})$  defined by \eqref{eq:blocks}, and the corresponding functions $(f_{n_j})$ and blocking functions
\[g_j:=f_{n_j}\circ \cdots\circ f_{n_{j-1}+1}.\]
Also, we put $G_j=g_j\circ \cdots \circ g_1$, for $j=1,2,\ldots$.

Then for a given $\delta\in(0,1/2)$, we choose disjoint arcs $A_1=\{z\in\partial \D :{\rm Im}(z)\ge\delta\}$ and $A_2=\{z\in\partial \D : {\rm Im}(z)\le -\delta\}$, so that the distance from the arc $A_1$ to the arc $A_2$ is strictly greater than $\delta$.

It follows from \eqref{eq:mixtopent} that, for $j\ge 1$, $k=1,2$, and for all measurable sets $E\subset \partial \D$ with $|E|>0$, we have
\begin{equation}\label{eq:mixedgj}
\frac{|A_k \cap g_j^{-1}(E)|}{|E|}\ge \frac{|A_k|}{2\pi} - \frac14 = \frac{\pi-2\sin^{-1}\delta}{2\pi}-\frac14>0,
\end{equation}
since $\sin^{-1}\delta\le \pi\delta/2<\pi/4$.

Now let $m\in\N$. By applying \eqref{eq:mixedgj} in turn to $g_m$, $g_{m-1}$, \ldots, $g_1$ and the resulting preimage subsets of $A_1$ and $A_2$, we deduce that for any sequence $(k_0, k_1\ldots k_m)$, where $k_j\in\{1,2\}$, each set
\[
B_{k_0\ldots k_m}=A_{k_0}\cap G_1^{-1}(A_{k_1})\cap \cdots\cap G_m^{-1}(A_{k_m}),
\]
has positive measure. Note that $B_{k_0\ldots k_m}$ is the set of points in $\partial\D$ with initial itinerary $(k_0, k_1\ldots k_m)$, with respect to the pair $\{A_1,A_2\}$ and the forward composition sequence $g_j, j=1,\ldots,m$.

By the construction and the definitions of $A_1$ and $A_2$, we see that any pair of points in $\partial\D$ with distinct itineraries of this type with respect to $(g_j)$ must be $(m,\delta)$-separated. Since there are $2^{m+1}$ distinct itineraries of this type and each of the corresponding sets $B_{k_0\ldots k_m}$ is non-empty, we deduce that the sequence $(g_j)$ has an $(m,\delta)$-separated set consisting of at least $2^{m+1}$ points. Therefore the original forward iteration sequence $(f_{n})$, which was blocked to give the sequence $(g_j)$, has for each $m\in\N$ an $(n_m,\delta)$-separated set consisting of at least $2^{m+1}$ points. So, for the sequence $(f_n)$ we have $N(n_m,\delta)\ge 2^{m+1}$ and hence
\begin{align}
\limsup_{n\to\infty}\frac1n\log N(n,\delta)\ge &\limsup_{m\to\infty}\frac{1}{n_m}\log N(n_m,\delta)\\
&\ge \limsup_{m\to\infty}\frac{m+1}{n_m}\log 2\\
&\ge \frac{\log 2}{2\left( \frac{\log 4K}{\delta} +1\right)}\mu,
\end{align}
by \eqref{eq:mnmmu} in the proof of Theorem~C, part~(b).


\section{Topological entropy zero example: proof of Theorem~E}
Theorem~E states that if $(\lambda_n)$ is a sequence in $(0,1)$ such that $\sum_{n=1}^\infty (1-\lambda_n)<\infty$ and
\[
f_n(z):=z\frac{z+\lambda_n}{1+\lambda_n z}, \quad n \in\N,
\]
then the forward composition sequence $F_n:=f_n\circ\cdots\circ f_1$ has topological entropy zero on $\partial\D$.

First we need a lemma about the existence of points where $F_n=f_n\circ \cdots \circ f_1$ is convergent.
\begin{lem}\label{lem:cvgce}
There exists a number $\theta_0\in (0,\pi/2)$ such that
\begin{itemize}
\item[(a)]
for $0<\theta \le \theta_0$ and $n\ge 1$, $\arg F_n(e^{i\theta})\in (0,\pi)$ and is strictly increasing with $n$;
\item[(b)]
$F(e^{i\theta})=\lim_{n\to\infty} F_n(e^{i\theta})$ satisfies
\[
0 \le \arg F(e^{i\theta})\le \pi, \text{ for } 0\le \theta\le \theta_0, \quad\text{and}\quad F(e^{i\theta_0})=-1.
\]
\end{itemize}
\end{lem}
\begin{proof} We move the setting from~$\D$ to $\Hyp =\{z=x+iy:x>0\}$, via the mapping $\alpha(z)=(1+z)/(1-z)$. On $\partial \Hyp$  the estimates are simpler since they involve real intervals rather than circular arcs. The image under~$\alpha$ of the upper half of $\partial\D$ traversed anticlockwise is the positive imaginary axis traversed downwards. A calculation shows that $f_n$ in~$\D$ is conjugate to
\[
g_n(z)=(\alpha\circ f_n \circ \alpha^{-1})(z)=\tfrac12(1+\lambda_n)z+\tfrac12(1-\lambda_n)z^{-1},\quad z \in \Hyp.
\]
So, for $z=iy$, where $y>0$ we study the forward composition $\Phi_n=\phi_n\circ \cdots \circ \phi_1$, where
\[
\phi_n(y)=\tfrac12(1+\lambda_n)y-\tfrac12(1-\lambda_n)y^{-1}.
\]
The following three properties of $\phi_n$ will be used repeatedly.
First, notice that
 \begin{equation}\label{phinpositive}
\phi_n(y)>0 \quad \text{if and only if}\quad |y|>\sqrt{\frac{1-\lambda_n}{1+\lambda_n}}.
\end{equation}
Also, for $y>0$,
\begin{eqnarray}\label{Phinrecurrence}
\phi_n(y)= y - \tfrac12(1-\lambda_{n})(y+y^{-1}) \leq y,
\end{eqnarray}
and finally,
 \begin{equation}\label{phininduction}
\phi_n(y)\geq  \lambda_n y  \quad \text{if and only if}\quad y\geq 1,
\end{equation}
since we can write $\phi_n(y)=\lambda_n y + \frac12 (1-\lambda_n)(y-y^{-1})$, and $y-y^{-1}\geq 0$ if and only if $y\geq 1$.

Starting with $y>0$, it follows from \eqref{Phinrecurrence} that  the sequence $\Phi_n(y):=\phi_n\circ\cdots\circ\phi_1(y)$ is strictly decreasing as long as the sequence remains in $(0,\infty)$. We now define
\[
E:=\{y\in (0,\infty): \Phi_n(y)>0,\text{ for } n\ge 1\}\quad \text{and}\quad y_0=\inf E.
\]
Then $y_0\ge 0$.

Now suppose that $y>\left(\prod_{n=1}^\infty \lambda_n\right)^{-1} > 1$. We shall show by induction that
\[
\Phi_n(y)\ge \left(\prod_{k=1}^n \la_k\right)y,\quad\text{for } n\ge 1.
\]
This inequality holds for $n=1$ by \eqref{phininduction} since $y\ge 1$.
If it holds for some $n\ge 1$, and therefore $\Phi_n(y)\geq 1$, then, by \eqref{phininduction} again,
\[
\Phi_{n+1}(y) =\phi_{n+1}(\Phi_n(y)) \geq \la_{n+1} \Phi_n(y)\geq \la_{n+1} \left(\prod_{k=1}^n \lambda_k\right)y =
 \left(\prod_{k=1}^{n+1} \lambda_k\right)y,
\]
as required. It follows that~$E$ contains the interval $(\left(\prod_{n=1}^\infty \lambda_n\right)^{-1},\infty)$, so $y_0\le \left(\prod_{n=1}^\infty \lambda_n\right)^{-1}$.

Next, we observe that~$E$ is an interval, clearly an unbounded one. This holds because each function $\phi_n$ is increasing on $(0,\infty)$ and in particular on the interval $\left(\sqrt{(1-\lambda_n)/(1+\lambda_n)},\infty\right)$ where $\phi_n$ is positive (by \eqref{phinpositive}), so
\[
\Phi_n(y) <\Phi_n(y'), \quad \text{if } y<y' \text{ and } y\in E.
\]
To complete the proof, we need to show that $\lim_{n\to \infty} \Phi_n(y_0)=0$. It is clear by the continuity of the functions $\phi_n$ on $(0,\infty)$ and the fact that~$E$ is an interval that $\Phi_n(y_0)> 0$ for all~$n\ge 1$. Thus $\lim_{n\to \infty} \Phi_n(y_0)$ exists and is at least~$0$. We can now define
\[
\Phi(y):=\lim_{n\to\infty}\Phi_n(y),\quad\text{for } y\ge y_0.
\]

Suppose for a contradiction that $\Phi(y_0)>0$. To obtain a contradiction, we shall deduce from this that there must exist $y\in E$ such that $0<y<y_0$. First choose~$N$ so large that
\begin{equation}\label{PhiN-ineq1}
\sqrt{\frac{1-\lambda_n}{1+\lambda_n}}<\tfrac12\Phi(y_0), \quad\text{for } n\ge N,
\end{equation}
\begin{equation}\label{PhiN-ineq2}
\Phi(y_0)<\Phi_N(y_0)<2\Phi(y_0),
\end{equation}
and
\begin{equation}\label{PhiN-ineq3}
\tfrac12C_0\left(\sum_{k=N+1}^\infty (1-\lambda_k)\right)<\tfrac12\Phi(y_0),
\end{equation}
where
\[
C_0=\max\{t+1/t: \tfrac12\Phi(y_0)\le t\le 2\Phi(y_0)\}>0.
\]
By the continuity of $\Phi_N$, there exists $y\in(0,y_0)$ such that
\begin{equation}\label{eq:y-props}
\Phi(y_0)<\Phi_N(y)<\Phi_N(y_0).
\end{equation}
We claim that $\Phi_n(y)>\frac12\Phi(y_0)$ for $n>N$, which contradicts the definition of $y_0=\inf E$. Assuming that this inequality holds for some $n\ge N$, we deduce by \eqref{phinpositive}, \eqref{Phinrecurrence}, \eqref{PhiN-ineq1}, \eqref{PhiN-ineq2}, \eqref{PhiN-ineq3} and \eqref{eq:y-props}, that
\begin{align*}
\Phi_{n+1}(y)&=\Phi_N(y)-(\Phi_N(y)-\Phi_{N+1}(y)) - \cdots -(\Phi_{n}(y)-\Phi_{n+1}(y))\\
&>\Phi(y_0)-\tfrac12C_0\sum_{k=N+1}^{n+1}(1-\lambda_k) >\tfrac12 \Phi(y_0),
\end{align*}
which proves the claim, by induction (note that all the quantities $\Phi_n(y)$ on the right of the above equation for $\Phi_{n+1}(y)$ lie in the interval $\left[\tfrac12\Phi(y_0), 2\Phi(y_0)\right]$). This completes the proof of the lemma.
\end{proof}
Next we note that the maps $f_n$ are all locally expanding on $\partial\D$, though by varying amounts; the estimate in Lemma~\ref{lem:expanding} follows from \cite[Proposition~1]{Martin}, for example.
\begin{lem}\label{lem:expanding}
With the functions $f_n$ defined as in the statement of Theorem~E, we have
\[
|f'_n(\zeta)|\ge f'_n(1)= \frac{2}{1+\lambda_n}>1,\quad\text{for }\zeta\in\partial\D.
\]
\end{lem}
We say that a finite set of points on $\partial\D$ is $\delta$-\textit{separated}, where $\delta>0$, if all pairs of distinct points in the set lie at least~$\delta$ apart, where the distance is measured along $\partial\D$. Clearly any set of points in $\partial \D$ that is $\delta$-separated contains at most $2\pi/\delta$ points.

Recall from Section~1 that a set $S\subset \partial\D$ is $(n,\delta)$-separated, where $n\ge 0$ and $\delta > 0$, for $F_n=f_n\circ \cdots\circ f_1$ if any two distinct points $\zeta, \zeta' \in S$ satisfy
\begin{equation}\label{eq:maxdist}
\max_{0\le m\le n}|F_m(\zeta)-F_m(\zeta')| \ge \delta.
\end{equation}
Also recall that the quantity $N(n,\delta)$, where $n\ge 0, \delta>0$, denotes the maximum cardinality of any $(n,\delta)$-separated set for the system $(f_n)$. To prove Theorem~E, we shall estimate this quantity from above and show that, for all $\delta>0$ we have
\[
\limsup_{n\to\infty}\frac1n \log N(n,\delta) =0,
\]
from which it follows that $(f_n)$ has topological entropy zero.

First, we describe sequences of sets that will play key roles in our calculations. Let $\theta_0$ be as in Lemma~\ref{lem:cvgce}, put $A_0=\{e^{i\theta}:|\theta|\le \theta_0\}$ and, for $n\ge 1$, put
\[
A_n=F_n(A_0)\quad\text{and}\quad B_n = \partial \D\setminus A_n.
\]
The behaviour of each $f_n$ (and so $F_n$) on the lower half of $\partial\D$ is obtained by reflection from that of $f_n$ on the upper half. Therefore, by Lemma~\ref{lem:cvgce},
\begin{equation}\label{eq:AnBnprops}
A_n\subset A_{n+1}\subset \{e^{i\theta}:|\theta|<\pi\}, \text{ for } n\ge 0, \quad \bigcup_{n=0}^\infty A_n =\partial\D \setminus \{-1\},\quad \bigcap_{n=0}^\infty B_n=\{-1\}.
\end{equation}
Also, for $n=1,2,\ldots,$ we have
\[
B_{n-1} \supset B_n\quad\text{and}\quad f_n^{-1}(B_n)\subset B_{n-1},\text{ since } f_n(A_{n-1})=A_n.
\]

Therefore, all orbits of points in $F_n^{-1}(A_n)$ must have initial itineraries under $F_k$, $k=0,1\ldots,n,$ with respect to the pairs $(A_k,B_k)$, $k=0,1,\ldots,n,$ which are illustrated in Figure~1, with one of the following forms
\begin{equation}\label{eq:itin1}
A_0, A_1, \ldots, A_n,
\end{equation}
\begin{equation}\label{eq:itin2}
B_0,B_1, \ldots B_{m-1}, A_m, \ldots, A_n,\quad\text{where } 1\le m\le n,
\end{equation}
or
\begin{equation}\label{eq:itin3}
B_0, B_1, \ldots, B_n.
\end{equation}

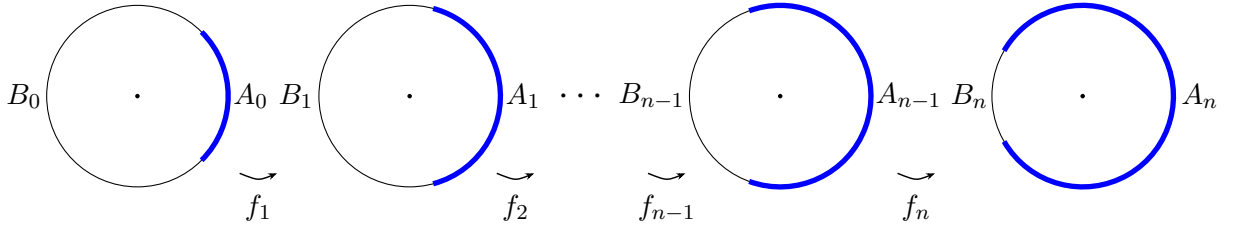
\begin{figure}[hbt!]
\centering
\begin{tikzpicture}[scale=1]
\tikzset{centerdot/.style={circle, fill=black, inner sep=0.6pt}}

\def\R{1.2}

\begin{scope}[shift={(0,0)}]
  \draw (0,0) circle (\R);
  \node[centerdot] at (0,0) {};

  \draw[line width=2pt, blue] (-45:\R) arc (-45:45:\R);



  \node[scale=1] at (0:1.5) {$A_0$};
  \node[scale=1] at (180:1.5) {$B_0$};

\end{scope}

\begin{scope}[shift={(3.6,0)}]
  \draw (0,0) circle (\R);
  \node[centerdot] at (0,0) {};

  \draw[line width=2pt, blue] (-75:\R) arc (-75:75:\R);



  \node[scale=1] at (0:1.5) {$A_1$};
   \node[scale=1] at (180:1.5) {$B_1$};
\end{scope}

\begin{scope}[shift={(8.5,0)}]
  \draw (0,0) circle (\R);
  \node[centerdot] at (0,0) {};

  \draw[line width=2pt, blue] (-110:\R) arc (-110:110:\R);


  \node[scale=1] at (0:1.7) {$A_{n-1}$};
  \node[scale=1] at (180:1.7) {$B_{n-1}$};
\end{scope}

\begin{scope}[shift={(12.5,0)}]
  \draw (0,0) circle (\R);
  \node[centerdot] at (0,0) {};

  \draw[line width=2pt, blue] (-150:\R) arc (-150:150:\R);

  \node[scale=1] at (0:1.55) {$A_n$};
  \node[scale=1] at (180:1.5) {$B_n$};
\end{scope}

\node at (5.9,0) {\Large $\cdots$};

\tikzset{
  map/.style={
    -{Stealth[length=1.2mm,width=0.9mm]},
    line width=0.55pt,
    shorten >=1.6pt,
    shorten <=1.6pt
  }
}

\def\y{-1.00}
\def\dy{0.12}

\newcommand{\tinyarrow}[3]{
  \draw[map] (#1,\y) .. controls ({(#1+#2)/2}, {\y-\dy}) .. (#2,\y)
    node[midway, below, yshift=-2pt] {#3};
}

\tinyarrow{1.30}{1.90}{$f_1$}

\tinyarrow{4.7}{5.3}{$f_2$}

\tinyarrow{6.70}{7.3}{$f_{n-1}$}

\tinyarrow{10.0}{10.6}{$f_n$}
\end{tikzpicture}
\caption{\small The sets $A_k$ (in blue) and their complements $B_k$, for $k=0,1,\ldots, n$}
\end{figure}

Now, let $\delta>0$ be fixed. By \eqref{eq:AnBnprops} we can take $K(\delta)\in\N$ so large that
\begin{equation}\label{eq:Bn}
|B_k|\le \tfrac12\delta,\quad \text{for } k \ge K(\delta).
\end{equation}

For $n\in\N$, we let~$S_n\subset\partial\D$ be any $(n,\delta)$-separated set for $(f_n)$. We will show that, for $n\ge K(\delta)$, the cardinality of $S_n$ is at most
\begin{equation}\label{eq:finalN}
2^{3K(\delta)-1}\left(\frac{\pi}{\delta}\right)^{K(\delta)} (n+2).
\end{equation}
This expression is an upper bound for $N(n,\delta)$ and hence, for all $\delta>0$ we have
\[
\limsup_{n\to\infty}\frac1n \log N(n,\delta) =0,
\]
which proves Theorem~E.

To obtain the upper bound \eqref{eq:finalN} for $N(n,\delta)$, we consider $n+2$ disjoint subsets of $S_n$, corresponding to the different initial itineraries in \eqref{eq:itin1}, \eqref{eq:itin2} and \eqref{eq:itin3}, and obtain upper bounds for the cardinality of each of them. We define
\[
S_n(A_0)=\{\zeta\in S_n: F_k(\zeta)\in A_k, \text{ for } k=0,1,\ldots, n\}
\]
and, for $m=1,\ldots, n$,
\[
S_n(A_m):=\{\zeta\in S_n:F_k(\zeta) \in B_k, \text{ for } k=0,1,\ldots, m-1, \text{ and } F_m(\zeta)\in A_m\}.
\]
It is clear that every point of $S_n\cap F_n^{-1}(A_n)$ lies in exactly one of the sets $S_n(A_m), 0\le m\le n$. We also define
\[
S_n(B_n):=\{\zeta\in S_n:F_k(\zeta) \in B_k, \text{ for } k=0,1,\ldots, n\}.
\]

For $m=0,1,\ldots, n$, we let $N_n(m)$ denote the maximal cardinality of any $(n,\delta)$-separated subset of $S_n(A_m)$ and let $N_n$ denote the maximal cardinality of any $(n,\delta)$-separated subset of $S_n(B_n)$.

Since $f_k$ maps $A_{k-1}$ injectively onto $A_k$, for $k=1,\ldots, n$, and is expanding there, the set $F_n(S_n(A_0))$ must be $\delta$-separated, and hence
\begin{equation}\label{eq:Nn0}
N_n(0)\le \frac{2\pi}{\delta}.
\end{equation}

Suppose from now on that $n\ge K(\delta)$. To bound $N_n(m)$, for $m=1,\ldots,n$, we suppose that~$T$ is an $(n,\delta)$-separated set of points taken from $S_n(A_m)$, with cardinality~$N$, and show that~$N$ cannot be too large. Recall, that the initial itinerary of the points in $S_n(A_m)$ is of the form in \eqref{eq:itin2}. 

Since~$T$ has~$N$ points, there exists at least one arc of $\partial\D$ that is of length at most $\tfrac12 \delta$ and contains a subset of at least $\delta/(4\pi)N$ points from~$T$; call this subset~$T_0$. We assume that~$N$ is so large that all the lower bounds for  cardinalities of sets that appear in this and the following reasoning are at least~1.

Since $f_1$ has degree~2, we deduce that $f_1(T_0)$ has at least $\tfrac12 \delta/(4\pi)N$ distinct points. We can therefore choose a subset $T'_1$ of $f_1(T_0)$ that contains at least $\tfrac12 (\delta/(4\pi))^2 N$ points and lies entirely in an arc of $\partial\D$ of length $\tfrac12\delta$. Let $T_1=T_0 \cap f_1^{-1}(T'_1)$. Then
\[
T_1\subset T_0\subset S_n(A_m) \text{ and } T_1 \text{ contains at least } \tfrac12 (\delta/(4\pi))^2 N\text { points}.
\]

Since $f_1(T_1) = T'_1$ and $f_2$ has degree~2, we deduce that $F_2(T_1)$ has at least $(\tfrac12)^2 (\delta/(4\pi))^2 N$ distinct points. We can therefore choose a subset $T'_2$ of $F_2(T_1)$ that contains at least $(\tfrac12)^2 (\delta/(4\pi))^3 N$ points and lies entirely in an arc of $\partial\D$ of length $\tfrac12\delta$. Let $T_2=T_1 \cap F_2^{-1}(T'_2)$. Then
\[
T_2\subset T_1\subset T_0\subset S_n(A_m) \text{ and } T_2 \text{ contains at least } (\tfrac12)^2 (\delta/(4\pi))^3 N\text { points}.
\]

Continuing in this way, for $1\le k\le \min\{m-1,K(\delta)-1\}$, we obtain sets $T_k$ such that
\begin{equation}\label{eq:cardTk}
T_{k}\subset T_{k-1}\subset S_n(A_m) \text{ and } T_k \text{ has at least } (\tfrac12)^{k} (\delta/(4\pi))^{k+1}N \text{ points},
\end{equation}
and
\[
|F_k(\zeta)-F_k(\zeta')|\le \tfrac12 \delta,\quad \text{for }\zeta,\zeta'\in T_{k}.
\]
Now suppose that $1\le m < K(\delta)$. Then, we have
\begin{equation}\label{eq:orbitsclose1}
\max\{|F_k(\zeta)-F_k(\zeta')|:0\le k\le m-1\}\le \tfrac12 \delta,\quad \text{for }\zeta,\zeta'\in T_{m-1}.
\end{equation}
Recall that $T_{m-1}\subset T\subset S_n(A_m)\subset S_n$, so $T_{m-1}$ is an $(n,\delta)$-separated set of points and $F_m(T_{m-1})\subset A_m$. Since $f_k$ maps $A_{k-1}$ injectively onto $A_k$, for $k=1,\ldots, n$, and is expanding there, it follows from \eqref{eq:orbitsclose1} that
\[
|F_n(\zeta)-F_n(\zeta')| \ge \delta, \quad \text{for }\zeta,\zeta'\in T_{m-1};
\]
that is, the points $F_n(\zeta),$ where $\zeta\in T_{m-1},$ are distinct and form a $\delta$-separated set. However, any $\delta$-separated set in $\partial\D$ has at most $2\pi/\delta$ elements, so we deduce, from \eqref{eq:cardTk} with $k=m-1$, that
\[
\left(\frac12\right)^{m-1} \left(\frac{\delta}{4\pi}\right)^{m}N\le\frac{2\pi}{\delta},
\]
giving the bound
\begin{equation}\label{eq:Nnm1}
N_n(m)\le 2^{3m}\left(\frac{\pi}{\delta}\right)^{m+1} \le 2^{3(K(\delta)-1)}\left(\frac{\pi}{\delta}\right)^{K(\delta)},\quad\text{for }1\le m\le K(\delta)-1.
\end{equation}
On the other hand, if $m \ge K(\delta)$, then  by the construction of the sets $T_k\subset S_n(A_m)$, $k=0,1,\ldots, K(\delta)-1,$ and \eqref{eq:Bn}, we deduce that
\[
\max\{|F_k(\zeta)-F_k(\zeta')|:0\le k\le m-1\}\le \tfrac12 \delta,\quad \text{for }\zeta,\zeta'\in T_{K(\delta)-1}.
\]
Arguing as above, we deduce that
\begin{equation}\label{eq:Nnm2}
N_n(m)\le 2^{3(K(\delta)-1)}\left(\frac{\pi}{\delta}\right)^{K(\delta)},\quad\text{for } K(\delta) \le m\le n.
\end{equation}
Therefore, whenever $1\le m\le n$, we have
\begin{equation}\label{eq:Nnm}
N_n(m)\le 2^{3(K(\delta)-1)}\left(\frac{\pi}{\delta}\right)^{K(\delta)}.
\end{equation}

Next, the maximum cardinality $N_n$ of any  $(n,\delta)$-separated set of points in $S_n(B_n)$ can be shown to satisfy a similar upper bound by supposing that~$T$ is an $(n,\delta)$-separated set of points taken from $S_n(B_n)$, with cardinality~$N$, and making a similar construction of nested subsets $T_k$, $0\le k\le K(\delta)-1$, of $T$. As above, we deduce that $T_{K(\delta)-1}\subset T\subset S_n(B_n)\subset S_n$, so $T_{K(\delta)-1}$ is an $(n,\delta)$-separated set of points with at least $(\tfrac12)^{K(\delta)-1}(\delta/(4\pi))^{K(\delta)} N$ points such that
\[
\max\{|F_k(\zeta)-F_k(\zeta')|:0\le k\le K(\delta)-1\}\le \tfrac12 \delta,\quad \text{for }\zeta,\zeta'\in T_{K(\delta)-1},
\]
and $F_{K(\delta)}(T_{K(\delta)-1})\subset B_{K(\delta)}$.

In view of \eqref{eq:Bn}, any $\delta$-separated set in $B_K(\delta)$ contains at most one point. Thus,
\begin{equation}\label{eq:Nn}
\left(\frac12\right)^{K(\delta)-1} \left(\frac{\delta}{4\pi}\right)^{K(\delta)}N \le 1, \quad\text{so}\quad N_n
\le 2^{3K(\delta)-1}\left(\frac{\pi}{\delta}\right)^{K(\delta)}.
\end{equation}

Combining the bounds for the cardinality of the $n + 2$ subsets $S_n(A_m)$, $m=0,1,\ldots, n,$ and $S_n(B_n)$, using \eqref{eq:Nn0}, \eqref{eq:Nnm1}, \eqref{eq:Nnm2} and \eqref{eq:Nn}, we deduce that the maximum cardinality of $S_n$ is at most
\[
N_n(0)+N_n(1) +\cdots + N_n(m) + N_n \le 2^{3K(\delta)-1}\left(\frac{\pi}{\delta}\right)^{K(\delta)} (n+2),
\]
which gives \eqref{eq:finalN}. This completes the proof of Theorem~E.

\section*{Acknowledgements} The authors thank
\begin{itemize}
\item Anna Miriam Benini for many helpful conversations;
\item the Isaac Newton Institute for Mathematical Sciences, Cambridge, for support and hospitality during the programme {\it Complex dynamics: interactions and influences}, where work on this paper was undertaken, supported by EPSRC grant no EP/Z000580/1;
\item the Departament de Matem\`atiques i Inform\`atica, Universitat de Barcelona, for support and hospitality while this work was being developed.
\end{itemize}

\section*{Declarations} The authors declare that they have no financial and personal relationships with other people or
organisations that could inappropriately influence this work. Also, no data was collected in the production of this work.

The authors' sole use of generative~AI was to ask Copilot how to write TikZ code for the diagram.

\bibliography{Wandering2-ent}

\hspace{-0.3cm}{\em Emails}:\\vasiliki.evdoridou@open.ac.uk,  nfagella@ub.edu, phil.rippon@open.ac.uk, gwyneth.stallard@open.ac.uk.

\hspace{-0.3cm}{\em ORCID iDs}:\\
Vasiliki Evdoridou: 0000-0002-5409-2663\\
N\'uria Fagella: 0000-0002-5466-0579\\
Phil Rippon: 0000-0003-4844-1190\\
Gwyneth Stallard: 0000-0001-7621-2219

\end{document}